\documentclass[12pt]{amsart}
\usepackage{amssymb, amsmath}
\usepackage{graphics}
\usepackage{epsfig}
\usepackage{amsxtra}
\usepackage{color}

\numberwithin{equation}{section}

\def\bbbone{{\mathchoice {1\mskip-4mu {\rm{l}}} {1\mskip-4mu {\rm{l}}}
{ 1\mskip-4.5mu {\rm{l}}} { 1\mskip-5mu {\rm{l}}}}}

\newtheorem{thm}{Theorem}[section]

\newtheorem{lemma}[thm]{Lemma}

\newtheorem{prop}[thm]{Proposition}

\theoremstyle{definition}

\newcommand{\defeq}{\stackrel{\rm{def}}{=}}

\renewcommand{\Re}{\operatorname{\rm Re}\nolimits}
\renewcommand{\Im}{\operatorname{\rm Im}\nolimits}

\def \supp {\operatorname{supp}}

\def \supp {\operatorname{supp}}
\def \tr {\operatorname{tr}}

\def \mco {{\mathcal O}}

\def \vol {{\rm vol}}

\def \restrict {\upharpoonright}

\def \mcd {{\mathcal D}}
\def \mch {{\mathcal H}}
\def \Real {{\mathbb R}}

\def \Sphere {\mathbb{S}}
\def \Complex {\mathbb{C}}
\def \Natural {{\mathbb N}}
\def \Sphere {{\mathbb S}}

\def \mcr{{\mathcal R}}

\def \U+ {U_+}
\def \Integers {{\mathbb Z}}

\title
 [Sharp lower bound]
{
A sharp lower bound for a resonance-counting function in even dimensions
}
   \author { T.J. Christiansen}
\address{Department of Mathematics,
University of Missouri,
Columbia, Missouri 65211, USA} 
\email{christiansent@missouri.edu}

\begin{document}

\begin{abstract}  This paper proves sharp lower bounds on a resonance counting function
for obstacle scattering in even-dimensional Euclidean space without 
a need for trapping assumptions. Similar lower bounds
are proved for some other compactly supported perturbations of
$-\Delta$ on $\Real^d$, for example, for the Laplacian 
for certain metric perturbations 
on $\Real^d$.  
The proof uses a Poisson formula for resonances, complementary to one 
proved by Zworski in even dimensions.
\end{abstract}

\maketitle
\section{Introduction}

The purpose of this paper is to prove a sharp polynomial lower bound on 
a resonance-counting function for certain compactly supported 
perturbations of $-\Delta$ on $\Real^d$ for even dimensions 
$d$.  The operators we consider include, for example, the Laplacian on the 
exterior of a bounded obstacle in $\Real^d$
and the Laplacian for many compactly supported metric perturations of
$\Real^d$.  These lower bounds complement 
 the sharp upper bounds proved
by Vodev \cite{vodeveven, vodev2}.  
The lower bounds of this paper do not require any trapping conditions. 
In order to prove the result, we prove 
a Poisson formula for resonances in even 
dimensions, complementary to that proved by Zworski
in \cite{zworskieven}.  The Poisson formula is valid for a large class of operators which are ``black-box'' perturbations of $-\Delta$ on $\Real^d$
as defined in \cite{sj-zw}.   Here $\Delta \leq 0$ is the Laplacian on $\Real^d$.

We begin with the classical problem of obstacle scattering.
Let $P$ denote $-\Delta$ with Dirichlet, Neumann, or Robin
boundary conditions on $\Real^d \setminus \overline{\mco}$, where
$\mco \subset \Real^d$ is a bounded open set with smooth boundary.  
For $0<\arg \lambda<\pi$ set $R(\lambda)=(P-\lambda^2)^{-1}$.  Then
it is well known that if $\chi \in C_c^\infty(\Real^d)$ then
$\chi R(\lambda)\chi$ has a meromorphic continuation to $\Complex$ if
$d$ is odd and to $\Lambda$, the logarithmic cover of 
$\Complex \setminus \{ 0\}$, if $d$ is even.  If $\chi$ is chosen
to be $1$ on an open set containing $\overline{\mco}$, the location of the 
poles of $\chi R(\lambda) \chi$ is independent of the choice of $\chi$.
We denote the set of all such poles, repeated according to multiplicity,
by $\mcr$.

We can describe a point  $\lambda \in \Lambda$ by its norm $|\lambda|$ 
and argument $\arg \lambda.$  In even dimensions
 we use the resonance counting functions defined as 
$$n_k(r)\defeq\# \{ \lambda_j\in \mcr:\; |\lambda_j|<r\; \text{and $ k\pi<\arg 
\lambda_j<(k+1)\pi$}\}, \; \text{for}\; k\in \Integers.$$
By the symmetry of resonances for self-adjoint operators 
(cf. \cite[Proposition 2.1]{ch-hi4}),
 $n_{-k}(r)=n_k(r)$.

\begin{thm}\label{thm:lowerbd}
Let $d$ be even and let $\mco \subset \Real^d$ be a bounded open set with 
smooth compact boundary $\partial \mco$. 
Assume in addition that $\mco \not = \emptyset$
and $\Real^d \setminus \overline{\mco}$ is 
connected.  Let $P$ be $-\Delta$ on 
$\Real^d \setminus \overline{\mco}$ with Dirichlet, Neumann,
 or Robin boundary conditions.
Then there is a constant $C_0>0$ so that 
\begin{equation}
\label{eq:uplowbd}
r^d/C_0\leq n_{-1}(r) \leq C_0 r^d\; \text{when $r\gg 1$}.
\end{equation}
\end{thm}
The upper bound $n_{-1}(r)\leq C_0 r^d$, $r\gg 1$,
 is due to Vodev \cite{vodeveven, vodev2}.  It is the 
lower bound which is new here.  For the Robin boundary condition, 
a function $v$ in the domain of $P$ must satisfy 
$\partial _\nu v =\gamma v$ on $\partial(\Real^d\setminus \overline{\mco})$,
where $\partial_\nu$ is the inward unit normal vector field and 
$\gamma \in C^\infty (\partial\mco;\Real)$ is a fixed function. 
Choosing $\gamma\equiv 0$ gives the Neumann boundary condition. 
The condition that $\Real^d \setminus \overline{\mco}$
be connected is not  necessary, but makes the proof cleaner since
we do not have infinitely many positive eigenvalues of $P$
as we would if $\Real^d \setminus \overline{\mco}$ had a bounded 
component. 

In {\em odd} dimension $d$, where $\mcr \subset \Complex$,
 we introduce the counting funcion
$$n_{odd}(r)=\# \{ \lambda_j\in \mcr: |\lambda_j|<r\}.$$
In  odd dimension $d\geq 3$, for obstacle scattering the 
analogous upper 
bound, $n_{odd}(r)\leq Cr^d$ for $r\gg 1$ of (\ref{eq:uplowbd}) is 
due to Melrose \cite{melroseobstacle}, but the known 
lower bound for general obstacles $\mco$ is  weaker: $n_{odd}(r) \geq c r^{d-1}$
for some $c>0$ and for sufficiently large $r$ \cite{beale,l-pdm}.

In either even or odd dimensions, Stefanov \cite[Section 4]{stefanovqm}
proved lower bounds on $n_{odd}(r)$ and $n_{-1}(r)$ 
proportional to $r^d$ under certain trapping assumptions on the geometry of 
$\Real^d \setminus \mco$.  On the other hand, again in either even or odd
dimensions, 
 for a
class of strictly convex $\mco$  asymptotics
of the number of resonances (of order $r^{d-1}$)
in certain regions are known, \cite{jin, sj-zwconvex}.  
In  odd dimensions  asymptotics of the resonance
counting function have been proved in the special case of $\mco$ equal to a ball
\cite{stefanov,zwrp}.

The primary result of \cite{chobstacle} is that in even dimension $d$, 
 under the hypotheses of Theorem \ref{thm:lowerbd} (with some
additional restrictions for the Robin boundary condition), if $k\in \Integers
 \setminus  \{ 0\}$, then
$\lim\sup_{r\rightarrow \infty}\log n_k(r)/\log r =d$.  
For $k=-1$  this is weaker than the lower bound of 
Theorem \ref{thm:lowerbd}.  Moreover, the proofs in 
\cite{chobstacle} use results of \cite{beale,l-p} for the particular
case of the operator $P$ of Theorem \ref{thm:lowerbd}, and those 
results do not obviously generalize to the setting of 
Theorem \ref{thm:lowerbd2}.  However, the techniques of this 
paper do not seem to give results for $n_k(r)$, $k\not = \pm 1$.

Obstacle scattering, as considered in Theorem \ref{thm:lowerbd},
forms a canonical class of scattering problems.  However, our 
result and its proof can easily be extended to a larger class of operators.
 In even dimension $d$, for the operator $P$ 
defined below, and for the much larger class of operators of the 
``black-box'' type of \cite{sj-zw}, the resolvent $(P-\lambda^2)^{-1}$
has a meromorphic continuation to $\Lambda$, \cite{sj-zw}.  Hence
 one can define resonances, the set $\mcr$,  and the 
resonance counting functions $n_k(r)$ just as for the Laplacian on the
exterior of a compact set.
\begin{thm}\label{thm:lowerbd2} Let $d$ be even, and let $(M,g)$ be a 
smooth, connected, noncompact $d$-dimensional
Riemannian manifold, perhaps with 
 smooth compact boundary $\partial M$.  Suppose
there is a compact set $K\subset M$ and an $R_0>0$ so that $M\setminus K$
is diffeomorphic to $\Real^d\setminus \overline{B}(0;R_0)$, and 
that $g$ restricted to $M\setminus K$ agrees with the flat metric on $\Real^d$.
Then let $P=-\Delta_g $ on $M$, with Dirichlet or Robin boundary conditions
if $\partial M \not = \emptyset$.  In the special case of 
$M=\Real^d \setminus \mco$ with $\mco$ as in Theorem \ref{thm:lowerbd}
and metric agreeing with the Euclidean metric outside a compact set,
we may choose $P=-\Delta_g +V$ with Dirichlet or Neumann boundary
conditions, for some $V\in C_c^\infty(\Real^d \setminus \mco;\Real)$.  Then if 
$\vol (K) \not = \vol (B(0;R_0))$, then there is a constant $C_0>0$ so
that 
$$r^d/C_0 \leq n_{-1}(r) \leq C_0 r^d\; \text{for }\; r\gg 1.$$
\end{thm}
In the statement of the theorem, $B(a;R)=\{ x\in \Real^d: \; |x-a|<R\}$,
and $\Delta_g\leq 0$ is the Laplacian on $(M,g)$. 
The operators $P$ 
defined in Theorem \ref{thm:lowerbd2} are examples of ``black box'' operators
as defined by Sj\"ostrand-Zworski \cite{sj-zw}, as recalled in 
Section \ref{ss:setup}.

Again, it is the lower bound of Theorem \ref{thm:lowerbd2}
which is new, as the upper bound is due to 
\cite{vodeveven, vodev2}.
 We shall  prove a more general
result, Theorem \ref{thm:lowerbd3}, from which Theorems 
\ref{thm:lowerbd} and \ref{thm:lowerbd2}
follow.

Tang \cite{tang} showed that for non-flat, compactly 
supported perturbations of the Euclidean metric on $\Real^d$, $d=4,\; 6$, 
the associated Laplacian has infinitely many resonances.  Under
certain geometric conditions one can prove the 
existence of many resonances for operators of the type
considered in Theorem \ref{thm:lowerbd2}.  In addition to 
the references already cited, we mention \cite{popov, zworskieven}
and references therein.

The proof of the lower bound of Theorems
\ref{thm:lowerbd} and \ref{thm:lowerbd2} uses
the wave trace, 
 informally given by 
\begin{equation}\label{eq:wavetraceinformal}
u(t) = 2\tr(\cos (t\sqrt{P})-\cos (t\sqrt{-\Delta})),
\end{equation}
and formally defined in (\ref{eq:wavetraceformal}).  For the 
operators of Theorem \ref{thm:lowerbd2}, because
we are in even dimension, the leading order singularity
of $u$
at  $0$ ``spreads out.''  That this, when
combined with a Poisson-type formula, can give good lower bounds
on similar resonance-counting functions was
proved in \cite{sj-zwlbII} and has been used, for example, in 
\cite{b-c-h-p,gu-zw,sj-zwlbII}.  In particular,
\cite{b-c-h-p,gu-zw} prove lower bounds analogous to (\ref{eq:uplowbd})
for certain even-dimensional manifolds hyperbolic near infinity.

 To use the result of \cite{sj-zwlbII}
requires a Poisson formula for resonances which is valid in any 
sufficiently small deleted neighborhood of $t=0$.  
Thus one of the main results of this paper is Theorem \ref{thm:poissonformula},
a Poisson formula for resonances
 in even dimensions. This result holds for a large class of 
``black-box'' perturbations (in the sense of \cite{sj-zw})
of $-\Delta$ on $\Real^d$, $d$ even.   Our result is complementary
to the results of  \cite{zworskieven}. 
See Section \ref{s:comparison} for  further discussion and references for
Poisson formulae in both even and odd dimensions.

We comment briefly on the structure of the paper.  Section 
\ref{s:prelimestimates} proves some bounds on the scattering
matrix which are needed later in the proof to control a term 
appearing in the Poisson formula.  In Section \ref{s:PF} we state
and prove the Poisson formula, Theorem \ref{thm:poissonformula}.  
Theorem \ref{thm:lowerbd3} is a more 
general version of the lower bound than Theorem \ref{thm:lowerbd2}.  In 
Section \ref{s:proofofthm} we prove Theorem \ref{thm:lowerbd3}, using
the Poisson formula and results from \cite{sj-zwlbII}
and Section \ref{s:prelimestimates}.  We finish 
the proof of Theorems \ref{thm:lowerbd}  and \ref{thm:lowerbd2} by using
results on the singularity at $0$ of $u(t)$, e.g. \cite{hosf,ivrii, melrose},
and estimates on the cut-off resolvent on the positive real axis due
to Burq \cite{burq} and Cardoso-Vodev \cite{c-v}.

\vspace{2mm}
\noindent
{\bf Acknowledgements.}  It is a pleasure to thank Maciej Zworski for helpful
conversations, one of which inspired this paper, and for comments on an earlier
version of the manuscript.

\subsection{Set-up and notation} \label{ss:setup}

We briefly introduce some notation which we shall use throughout the paper.

We recall the black-box perturbations of \cite{sj-zw}, using notation
as in that paper.  Let $\mch$ be a complex Hilbert space with orthogonal
decomposition
$$\mch= \mch_{R_0}\oplus L^2(\Real^d\setminus B(0;R_0))$$
where $B(0;R_0)\subset \Real^d$ is the ball of radius $R_0$ 
centered at the orgin.  Let ${\bbbone}_{\Real^n\setminus B(0;R_0)}
:\mch \rightarrow L^2(\Real^d\setminus B(0;R_0))$ denote orthogonal
projection.   Let $P:\mch\rightarrow \mch$ be a linear 
self-adjoint operator  with domain $\mcd \subset \mch$.
We assume ${\bbbone}_{\Real^n\setminus B(0;R_0)} \mcd = H^2 (\Real^d\setminus B(0;R_0))$ and ${\bbbone }_{\Real^n\setminus B(0;R_0)} P = -\Delta 
\restrict_{\Real^n\setminus B(0;R_0)}$.  Moreover,
$P$ is lower semi-bounded  and  there is a $k_0\in \Natural$ so that 
${\bbbone}_{\Real^n\setminus B(0;R_0)} (P+i)^{-k_0}$ is trace class.

Under these assumptions on $P$, we may carefully define the wave trace
given informally by (\ref{eq:wavetraceinformal}).  Thus
\begin{multline}\label{eq:wavetraceformal}
u(t) \defeq 2 \left[ \tr_{\mch}\left( \cos (t\sqrt{P}) - {\bbbone}_{\Real^n\setminus B(0;R_0)} 
\cos (t \sqrt{-\Delta}){\bbbone}_{\Real^n\setminus B(0;R_0)} 
\right)  \right. \\ \left. - \tr_{L^2(\Real^d)} \left( {\bbbone}_{ B(0;R_0)} 
\cos (t \sqrt{-\Delta}){\bbbone}_{B(0;R_0)}
\right)\right].
\end{multline}
The factor of $2$ is included to be consistent with 
the definition of $u$ given via the wave group; see, for example,
\cite{melrosetrace,zworskieven}.

Let $P$ be a black-box perturbation of $-\Delta$ on $\Real^d$, and,
for $0<\arg \lambda<\pi$, set $R(\lambda)=(P-\lambda^2)^{-1}$. 
We denote by $R(\lambda)$ the meromorphic continuation to $\Lambda$, 
the logarithmic cover of $\Complex \setminus \{0\}$, if
$d$ is even.  If $d$ is odd, the continuation is to $\Complex$, and again
we denote it $R(\lambda)$.
We use 
$\mcr$ to denote all the  poles of continuation of the
resolvent $R(\lambda)$, repeated according to multiplicity.  We explicitly
include both those poles corresponding to eigenvalues and those which
do not.  

A point $\lambda\in \Lambda $ can be identified by specifying both its
norm $|\lambda|$ and argument $\arg \lambda$ where we do not 
identify points in $\Lambda$ whose arguments differ by nonzero integral
multiples of $2\pi$.  For $\lambda \in \Lambda$, we denote
$\overline{\lambda}= \overline{ |\lambda |e^{i\arg \lambda}}= 
|\lambda|e^{-i\arg 
\lambda}$.
For $k\in \Natural$, set 
$$\Lambda_k=\{ \lambda \in \Lambda: k\pi< \arg \lambda< (k+1)\pi\}.$$

Throughout this paper, $C$ denotes a positive constant, the value of 
which may change from line to line without comment.

\section{Preliminary estimates on the scattering matrix and related quantities}
\label{s:prelimestimates}

We shall need some estimates on the determinant of the scattering matrix
and on $\|S(\lambda)-I\|_{\tr}$.  The proof uses  a representation for
the scattering matrix from \cite{pe-zwsc} which we recall for the 
reader's convenience.  We remark that there are a number of
related representations in the literature; see, for 
example, \cite[Section 2]{pe-st} or \cite[Section 3]{zworskieven}.
\begin{prop}(\cite[Proposition 2.1]{pe-zwsc})\label{p:pe-zw}   
For $\phi\in C_c^{\infty}(\Real^d)$,
 let us denote by 
\begin{equation*}
{\mathbb E}^{\phi}_{\pm}(\lambda) :L^2(\Real^d)\rightarrow L^2(\Sphere^{d-1})
\end{equation*}
the operator 
\begin{equation}\label{eq:edef}
\left( {\mathbb E}^{\phi}_{\pm}(\lambda)f\right)(\theta)=\int_{\Real^d}
f(x)   \phi(x) \exp(\pm i \lambda \langle x,\omega \rangle ) dx .
\end{equation}  Let 
 $\chi_i\in C_c^{\infty}(\Real^d)$, $i=1,\; 2,\;3$ be
 such that $\chi_i\equiv 1$ near $\overline{B}(0;R_0)$
and $\chi_{i+1}\equiv 1$ on $\supp \chi_i$.  

Then for $0<\arg \lambda <\pi$,  $S(\lambda)= I+A(\lambda)$, where 
$$A(\lambda) = i \pi (2\pi)^{-d} \lambda^{d-2} {\mathbb E}^{\chi _3}_+ (\lambda) [\Delta, \chi_1]
R(\lambda) [\Delta, \chi_2] ^t{\mathbb E}^{\chi_3}_-(\lambda)$$
where $^t{\mathbb E}$ denotes the transpose of ${\mathbb E}$.  The identity
holds for $\lambda \in \Lambda$ by analytic continuation.
\end{prop}

Maciej Zworski suggested the proof of the following lemma.  This 
follows techniques of \cite{melroseobstacle,vodevsb, zwspb}.  
We note that results of Burq \cite{burq} and Cardoso-Vodev \cite{c-v}
show that there
are a large class of examples of black-box operators $P$ for which
(\ref{eq:resolventdecay}) holds.
\begin{lemma} \label{l:realbd}
Let the dimension $d\geq 2$ be even or odd, and let $P$ denote
an operator satisfying the general black-box conditions of \cite{sj-zw}
recalled in Section \ref{ss:setup}.  Assume that there is 
an $R_1>R_0$ so that for all $b>a>R_1$, if $\chi \in C_c^\infty(\Real^d)$
has support in $\{x\in \Real^d : a<|x|<b\}$ then there is a constant
$C$ depending on $P$ and $\chi$ so that 
\begin{equation} \label{eq:resolventdecay}
\| \chi R(\tau) \chi\| \leq C/\tau\; \text{ for $\tau
\in  (1,\infty)$ (that is,  $\arg \tau=0$)}.
\end{equation}
Let $S(\lambda)$ denote the  scattering matrix for $P$, unitary when 
$\arg \lambda =0$.  Then, for $\tau\in (1,\infty)$, 
there is a constant $C$ so 
that $\| I-S(\tau) \|_{\tr} \leq C (1+ \tau)^{d-1}$ where $\| \cdot \|_{\tr}$ 
is the trace class norm.
\end{lemma}
\begin{proof}
We note first that this is a result about large $\tau$ behavior, as $\lim
_{\tau \downarrow 0} S(\tau)=I$, see \cite[Section 6]{chobstacle}.

  For a bounded linear operator $B$, we denote by 
$s_1(B) \geq s_2(B)\geq ...$ the singular values of $B$.

We  use  the representation of $S(\tau)$ recalled in Proposition 
\ref{p:pe-zw}.  
We have  
$$s_j(S(\tau)-I) \leq 
\| i \pi (2\pi)^{-d} \tau^{d-2} {\mathbb E}^{\chi _3}_+\|_{L^2 \rightarrow
L^2} \| [\Delta, \chi_1]
R(\tau) [\Delta, \chi_2]\| _{L^2 \rightarrow L^2}
s_j( ^t{\mathbb E}^{\chi_3}_-(\lambda) ).$$
We may choose $\chi_1,\; \chi_2$ so 
that by  (\ref{eq:resolventdecay}), 
$$\|  [\Delta, \chi_1]
R(\tau) [\Delta, \chi_2]\| _{L^2 \rightarrow L^2}  \leq C(1 + \tau).$$
Moreover, $\|  i \pi (2\pi)^{-d} \tau^{d-2} {\mathbb E}^{\chi _3}_+\|_{L^2 \rightarrow
L^2}\leq C\tau^{d-2}$ since $\tau \in (0,\infty)$.
From \cite[(3.4.14)]{dy-zw}, which follows \cite{zwspb}, 
\begin{equation*}
s_j(^t{\mathbb E}^{\chi_3}_-(\tau) ) 
\leq C_1 \exp(C_1\tau - j^{\frac{1}{d-1}}/C_1).
\end{equation*}
Summarizing, 
\begin{equation}
\label{eq:hardsjbd}s_j(S(\tau)-I) \leq C (\tau^{d-1}+1)\exp(C_1\tau - j^{\frac{1}{d-1}}/C_1).
\end{equation}
However, since for $\tau \in (0,\infty)$, $S(\tau)$ 
is unitary we have that $s_j(S(\tau)-I)\leq 2$ for all $j$.  Using this 
and
(\ref{eq:hardsjbd}), by choosing $C_2$ sufficiently large, we have
for $\tau\in (0,\infty)$
\begin{equation}
s_j(S(\tau)-I)
\leq \left\{ \begin{array}{ll}
2 & \text{if $j\leq C_2\tau^{d-1}$}\\
C j^{-\frac{d+1}{d}}& \text{if $j\geq C_2 \tau^{d-1}$.}
\end{array}
\right.
\end{equation}
Thus, for $\tau \in (1,\infty)$
\begin{align*}
\|S(\tau)-I\|_{\tr} & = \sum s_j( S(\tau)-I) \\ & 
= \sum_{j\leq C_2\tau^{d-1}} s_j( S(\tau)-I)  + \sum_{j> C_2\tau^{d-1}}
s_j( S(\tau)-I)\\ & \leq 2 C_2 \tau^{d-1} + \sum_{j\geq C_2 \tau^{d-1}}C j^{-\frac{d+1}{d}}
 \leq 2 C_2 \tau^{d-1} + C_3 .
\end{align*}
\end{proof}

Here we use Lemma \ref{l:realbd} to prove Lemma \ref{l:argchange}.
This argument  has the advantage of 
working for a large class of perturbations $P$ of $-\Delta$.  However,
for the special case of the Laplacian on the exterior of an obstacle, 
a portion of Lemma \ref{l:argchange} has been proved in \cite{chobstacle} using
instead of Lemma \ref{l:realbd}
``inside-outside duality'' results of Eckmann-Pillet 
\cite{e-p1} and Lechleiter-Peters \cite{le-pe}.  In this
special case, the remainder of Lemma \ref{l:argchange}
could be proved in a similar way.
\begin{lemma}\label{l:argchange} Let the dimension $d$ be even, and let $P$ denote
an operator satisfying the hypotheses of Lemma \ref{l:realbd} on $P$.
  Then there is a constant $C>0$ so
that for $\tau >0 $ (i.e., $\arg \tau =0$), 
$$1\leq | \det S(\tau e^{i\pi})| \leq C\exp (C \tau^{d-1})\; \text{and}\;
|\arg S(\tau e^{i\pi})-\arg S(e^{i\pi})| \leq C (\tau^{d-1}+1)$$
where the argument is chosen to depend continuously on $\tau \in (0,\infty)$.
\end{lemma}
\begin{proof} 
From \cite[Proposition 2.1]{ch-hi4}, $S(e^{i\pi} \tau) =2I-J S^*(\tau)J$,
where $(Jf)(\omega)=f(-\omega)$.  Hence 
\begin{equation}\label{eq:srelation}
\det S(e^{i\pi} \tau)
=\det(2I-S^*(\tau)).
\end{equation}  
  Using that $S(\tau)$ is unitary, we can write
$\det S(e^{i\pi} \tau) =\prod (2-e^{-i\theta_j(\tau)})$, where 
$\{ e^{i \theta_j(\tau)}\}$ are the eigenvalues of $S(\tau)$ other than
$1$, repeated with 
multiplicity, and $\theta_j(\tau)\in \Real$.  Since $|2-e^{-i\theta_j(\tau)}|\geq 1$, we see 
immediately that $|\det S(\tau e^{i\pi})|\geq 1$.  
Moreover, from Lemma \ref{l:realbd}, (\ref{eq:srelation}) and the estimate
$|\det (I+B)|\leq \exp(\|B\|_{\tr})$ we obtain
the upper bound on $|\det S(\tau e^{i\pi})|$.

Now we turn to the question of bounding $\arg S(\tau e^{i\pi})$.
Note first that $\theta_j$ can 
be chosen so that $e^{i\theta_j(\tau)}$ depends continuously on $\tau$, except,
perhaps, at points where $e^{i\theta_j(\tau)}$ approaches $1$.
Let $\log$ denote the principal branch of the logarithm.
Since $\Re (2-e^{-i\theta_j(\tau)})\geq 1$,
$\log (2-e^{-i\theta_j(\tau)})$ is well-defined and is a continuous function
of $\tau $ if $e^{i\theta_j(\tau)}$ is a continuous function of $\tau$.  Now
\begin{equation}\label{eq:expsum}
\det S(\tau e^{i\pi}) = \exp \left( \sum \log (2-e^{-i\theta_j(\tau)}) \right).
\end{equation}
Since if $\lim_{\tau\rightarrow \tau_0^{\pm}}e^{i\theta_j(\tau)}=1$, then 
$\lim_{\tau\rightarrow \tau_0^{\pm}}\log(2-e^{i\theta_j(\tau)})=0$, the sum $\sum \log (2-e^{-i\theta_j(\tau)})$ depends continuously on 
$\tau\in(0,\infty)$, so that 
$$\arg S(\tau e^{i\pi})-\arg S(e^{i\pi})= 
\sum \Im \log (2-e^{-i\theta_j(\tau)}) - \sum \Im \log (2-e^{-i\theta_j(1)}).$$
But 
\begin{multline*}
\left|\sum \Im \log (2-e^{-i\theta_j(\tau)}) \right| \leq \sum | \Im \log (2-e^{-i\theta_j(\tau)})| \\ 
\leq C \sum | (1-e^{-i\theta_j(\tau)})| \leq C \| S(\tau)-I\|_{\tr }
\leq C(1+\tau^{d-1})
\end{multline*}
by Lemma \ref{l:realbd}.
\end{proof}
We remark that the result about the 
argument in Lemma \ref{l:argchange} is not, in general, true in 
odd dimensions, where $\det S(\tau e^{i\pi})=\overline{\det S(\tau)}$
if $\tau >0$ so that the change in the argument may obey Weyl-type 
asymptotics as $\tau \rightarrow \infty$.

\section{The Poisson formula}\label{s:PF}
We begin with several complex-analytic results which will be helpful in proving the estimates we need on $\det S(\lambda)$ and related quantities.

We shall use the following lemma when working with 
functions holomorphic or meromorphic in a sector.  
\begin{lemma} \label{l:lbs}
Set $U=\{ z\in \Complex: \Re z>0\}$.  Suppose
$f$ is analytic on $U$ and that there exist constants
$C,\; p$, $p\geq 1$, so that $|f(z)| \leq C\exp (C|z|^p)$ 
for any $z\in U$.  Suppose in addition that $f$ is 
nowhere vanishing on $U$.  Then there is a function $g$ analytic 
on $U$ so that $\exp g(z)=f(z)$, $z\in U$.  Moreover, given
$\epsilon>0$ there is a constant $C_\epsilon$ so that 
$$|g(z)|\leq C_\epsilon |z|^{p+\epsilon}\; \text{if $|\arg z|<\pi/2-\epsilon$, 
$|z|>1$}.$$
\end{lemma}
\begin{proof}
The existence of an analytic $g$ so that $\exp(g)=f$ is immediate.
The bound on $|g|$ can be proved, for example, by using 
the representation \cite[Theorem 3.3]{govorov}
of a function of finite order analytic in an 
angle. 
\end{proof}

Note that the 
hypothesis $p\geq 1$ in Lemma \ref{l:lbs} is necessary.  For example, the function
$f(z)= e^{-z}$ satisfies the other hypotheses of Lemma
\ref{l:lbs} with $p=0$, but does not satisfy the conclusion of the 
lemma with $p=0$.  For more about related questions, see \cite[Chapter 1]{govorov}.

We recall that the canonical factors $E_p$ are given by
$$E_0(z)=1-z\; \text{and}\; E_p(z) =(1-z)\exp(z+z^2/2+\cdot\cdot\cdot+ z^p/p)\; \text{for}\; p\in\Natural.$$

For entire functions $f$, a stronger result than the following lemma
holds-- compare \cite[Theorem I.9]{levin}.  The following
lemma, while likely not sharp, will suffice for our needs.
\begin{lemma}\label{l:lbawayzero}
Set $U=\{ z\in \Complex: \Re z>0\}$.  Suppose
$f$ is analytic on $U$ and that there exist constants
$C,\; p$, $p\geq 1$, so that $|f(z)| \leq C\exp (C|z|^p)$ 
for any $z\in U$.  Suppose in addition that there is a constant 
$C$ so that 
$$\#\{ z_j: \; z_j\in U; f(z_j)=0\; \text{and}\; |z_j|<r\} \leq C(1+|z|^p)$$
where the zeros $z_j$ are repeated with multiplicity.  Then for any 
$\epsilon,\eta>0$ there is a constant $C_{\epsilon, \eta}$ so
that for any $R>2$
$$\ln |f(z)|\geq -C_{\eta, \epsilon}(1+R^{p+\epsilon}),\; 1\leq |z|\leq R,\;
|\arg z|<\pi/2-\epsilon$$
if $z$ lies outside a family of excluded disks, the sum of whose radii does
not exceed $\eta R$.  Moreover, if $\varphi(z)= \prod
E_{\lfloor p \rfloor}(z/z_j),$ then 
$$|f(z)/\varphi(z)|\leq  C_{\epsilon}(1+R^{p+\epsilon}),\;
 1\leq  |z|\leq R,\; |\arg z|<\pi/2-\epsilon/2.$$
\end{lemma}
\begin{proof}
  By standard estimates on 
canonical products, for $\epsilon>0$ the entire function $\varphi$ satisfies
$|\varphi(z)| \leq C_\epsilon(1+|z|^{p+\epsilon/2})$.  Hence, by 
\cite[Theorem I.9]{levin}, for $R>1$ and
any $\eta>0$, there is a constant $C_{\eta, \epsilon}$ so that 
$$\ln |\varphi(z)|\geq - C_{\eta, \epsilon}(1+R^{p+\epsilon/2}),\;  |z|\leq R$$
where $z$ in addition lies outside a family of disks the sum of whose
radii does not exceed $\eta R$.  

Now consider the function $f(z)/\varphi(z)$, which 
 is a nonvanishing analytic function on $U$.
Moreover, with perhaps a new constant $C_{\eta, \epsilon}$
$$|f(z)/\varphi(z)| \leq C_{\eta, \epsilon}(1+R^{p+\epsilon/2}),\;  
|z|\leq R,\; z\in U$$
if $z$ lies outside a family of disks the sum of whose
radii does not exceed $\eta R$.  Since $f/\varphi$ is analytic, by the
maximum principle if $R\gg 1$
$$|f(z)/\varphi(z)|\leq C_{\eta, \epsilon}(1+R^{p+\epsilon/2}),\;
  |z|\leq R,\; |\arg z|<\pi/2-\epsilon/2.$$
The reason for shrinking the size of the sector is the need, for 
every previously excluded point $z'$ (ie, every point in one of the originally
excluded disks) to have bounds
on $|f/\varphi|$ on a closed curve  lying inside $U$ and containing
$z'$ in its interior.  Now since $f/\varphi$ is nonvanishing on $U$ we
can apply Lemma \ref{l:lbs} to the function $(f/\varphi) (z^{1/(1-\epsilon/2)})$
to complete the proof.
\end{proof}

Let $S$ denote the scattering matrix, unitary for $\arg \lambda =0$,
associated to a self-adjoint black-box type operator $P$. 
Then
define
\begin{equation}\label{eq:f}
f(\lambda) =\frac{ \det S(\lambda)}{\det S(\lambda e^{i\pi})}.
\end{equation}

The proof of the following lemma closely resembles that of 
\cite[Proposition 6]{zwpfodd}
 and of a related result of \cite[Section 2]{zworskieven}.  We include
the proof here for the convenience of the reader and 
 because the need for working in sectors
of $\Complex$ rather than in balls in $\Complex$ means that 
Lemmas \ref{l:lbs} and \ref{l:lbawayzero} 
substitute for what can be done with Cartan's lemma and 
Caratheodory's inequality for the disk.  The proof mostly 
 follows \cite{zworskieven}, highlighting the points at which
Lemmas  \ref{l:lbs} and \ref{l:lbawayzero} are needed.
\begin{lemma}\label{l:frep}  Let $d$ be even, and let $f$ be as in (\ref{eq:f}).
If $n_{-1}(r)+n_{-2}(r) \leq C r^p$, $p\in \Natural$, set 
$$P_1(\lambda)=\prod_{\substack{\lambda_j\in \mcr\\ -\pi/2<\arg \lambda_j<3\pi/2}}
E_p(\lambda/\overline{\lambda_j}),\; 
P_2(\lambda)=\prod_{\substack{\lambda_j\in \mcr\\- 3 \pi/2<\arg \lambda_j<\pi/2}}
E_p(\lambda e^{i\pi}/\overline{\lambda_j}) $$
$$Q_1(\lambda)= \prod_{\substack{\lambda_j\in \mcr\\-3\pi/2<\arg \lambda_j<\pi/2}}
E_p(\lambda/\lambda_j),\; Q_2(\lambda) = \prod_{\substack{\lambda_j\in \mcr\\ - \pi/2<\arg \lambda_j< 3 \pi/2}}E_p(\lambda e^{i\pi}/\lambda_j)$$
Then for $-3\pi/2<\arg \lambda<\pi/2$, 
$$f(\lambda) = e^{g(\lambda)} \frac{P_1(\lambda)}{Q_1(\lambda)}
 \frac{Q_2(\lambda)}
{P_2(\lambda)}$$
where  $g(\lambda)$ is analytic in $-3\pi/2<\arg \lambda<\pi/2$ and
for some $p'\geq p, C>0$ we have
$$|g(\lambda)|\leq C (1+|\lambda|^{p'}),\; -4\pi/3<\arg \lambda<\pi/3.$$
\end{lemma}
\begin{proof}  We identify the region
$\{ \lambda \in \Lambda: \; -3\pi/2<\arg \lambda<\pi/2\}$ with 
$\Complex \setminus i[0,\infty)$. 

 The results of \cite{vodeveven,vodev2} show that with 
our assumptions on $P$,
there is a $p\in \Natural$ so that $n_{-1}(r)+ n_{-2}(r)=O(r^p)$ as 
$r\rightarrow \infty$.  Using 
this and the fact that $R(\lambda)$ has at most finitely many poles
in $\Lambda_{0}$, the functions $P_j$, $Q_j$, $j=1,2$, are well-defined entire
functions.
Since $S(\lambda)S^*(\overline{\lambda})=I$,
$\lambda_j$ is a pole of $\det S(\lambda)$ if and only if 
$\overline{\lambda_j}$ is a zero of $\det S(\lambda)$.   
The function 
$$h(z)\defeq f(\lambda) \frac{Q_1(\lambda)}{P_1(\lambda)}\frac{P_2(\lambda)}{Q_2(\lambda)}
$$ is an analytic, nowhere vanishing function if $-3\pi/2<\arg \lambda<\pi/2$,
see \cite[Theorem 4.5]{ch-hi4},
so the existence of a function $g$ so that $\exp(g)=h$ in
this region is immediate.  What needs to be proved is the polynomial bound
on $g$ in this region.

We use another representation of the scattering matrix, see 
\cite[Section 2]{zworskieven}, cf. \cite[Section 3]{gu-zw}.  The 
one described below can be deduced from 
Proposition \ref{p:pe-zw} using also \cite[Section 3]{sj-zw}.
The scattering matrix   $S(\lambda)= I+A(\lambda)$,
where 
$$A(\lambda)= c_d\lambda^{d-2} {\mathbb E}^{\phi}_+( \lambda) 
(I+K(\lambda,\lambda_0))^{-1}[\Delta,\chi]^t {\mathbb E}^{\phi}_-(\lambda)$$
where $\phi\in C_c^\infty(\Real^d)$ is $1$ on $\overline{B}(0;R_0)$ and
${\mathbb E}^{\phi}_{\pm}$ are as defined in (\ref{eq:edef}).  The operator
$K(\lambda,\lambda_0)$ is a compact operator, analytic on $\Lambda$, defined
in \cite[Section 3]{sj-zw}:
\begin{align*}
K(\lambda,  \lambda_0)= [\Delta,\chi_0]R_0(\lambda)(1-\chi_1)\chi_4-[\Delta,\chi_2]R(\lambda_0)\chi_1+\chi_2(\lambda_0^2-\lambda^2)R(\lambda_0)\chi_1\\
\phi,\chi_i\in C_c^\infty(\Real^d),\; \chi_i\equiv 1 \;\text{on}\;\overline{B}(0;R_0),\; 
\chi_i\chi_{i-1}=\chi_{i-1},\; \phi \chi_4=\chi_4
\end{align*}
where $R_0(\lambda)=(-\Delta-\lambda^2)^{-1}$ when $0<\arg \lambda<\pi$ and
is the holomorphic extension otherwise.   Here $\lambda_0$ is a point in
$ \Lambda$ with
$0<\arg \lambda_0<\pi$ and is chosen to ensure the invertibility of $I+K(\lambda_0,
\lambda_0)$.  Hence by the analytic Fredholm theory $(I+K(\lambda, \lambda_0))^{-1}$
is a meromorphic function on $\Lambda$.

The assumptions made on $P$ ensure that
 the operator $K(\lambda,\lambda_0)^{k_0}$ is
trace class.  We remark that we make no effort here to find the 
optimal value of $p'$ so that  the statement of the lemma holds.
By techniques of \cite{vodevsb, zwrp} (see also \cite{vodeveven}),
if $|\arg \lambda|\leq \theta_0$ there is
a natural number $m$ and a constant $C$ (depending on 
$\theta_0$, but not on $|\lambda|$), so that
\begin{equation}\label{eq:detest}
|\det(I+K(\lambda,\lambda_0)^{k_0})|\leq \det(I+|K(\lambda,\lambda_0)|^{k_0})
\leq C\exp (C|\lambda|^{m}).
\end{equation}
Moreover, the number of zeros of $|\det(I+K(\lambda,\lambda_0)^{k_0})$
(counted with multiplicity)
in the region $\{ \lambda \in \Lambda: \; |\arg \lambda|\leq \theta_0;\;
|\lambda|\leq r\}$ is $O(r^m)$ (\cite{vodeveven, vodev2}).
We apply this to the inequality (\cite[Theorem 5.1]{g-k}
$$\| (I+K(\lambda, \lambda_0))^{-1}\| 
\leq \| (I+\|K(\lambda, \lambda_0)\|)^{k_0-1}\; 
\frac{\det (I+|K(\lambda,\lambda_0)|^{k_0})}{|\det (I+K(\lambda,\lambda_0)^{k_0})|
}.$$ 

It is here we can see the need of Lemmas \ref{l:lbs} and \ref{l:lbawayzero}.
Using these lemmas and the estimate (\ref{eq:detest}), we see that
given $\eta,\epsilon>0$,  there is a constant $C_{\epsilon, \eta}$ so
that
for $|\arg \lambda|<\theta_0 -\epsilon$ and 
$1\leq |\lambda| \leq R$, outside a family
of excluded disks, the sum of whose radii does not exceed $\eta R$, 
$$\| (I+K(\lambda,\lambda_0))^{-1}\| \leq C_{\epsilon, \eta}
\exp(C (1+ R^{m+\epsilon})).$$
Using this as in \cite{zwpfodd}, one can show that, perhaps with new constant
$C_{\epsilon,\eta}$ 
\begin{equation}\label{eq:detest2}
|\det S(\lambda)| \leq C_{\eta,\epsilon} \exp C_{\eta,\epsilon}(R^{(m+\epsilon)d})
\end{equation}
for $\lambda$ in the same region.

Since $\det S(\lambda) Q_1(\lambda)$ is analytic for 
$-3\pi/2<\arg \lambda<\pi/2$, using (\ref{eq:detest2}) and the maximum principle
gives
$$
|Q_1(\lambda) \det S(\lambda)| \leq C_{\epsilon} \exp C_{\eta,\epsilon}(|\lambda|^{(m+\epsilon)d})
\; |\lambda|>1,\; -3\pi/2+\epsilon<\arg \lambda<\pi/2-\epsilon.$$
Since $Q_1(\lambda)\det S(\lambda) /P_1(\lambda)$ is a nowhere vanishing
analytic function in $-3\pi/2<\arg \lambda<\pi/2$, applying Lemmas
\ref{l:lbawayzero} and \ref{l:lbs} we find that 
$$Q_1(\lambda)(\det S(\lambda) )/P_1(\lambda)=\exp g_1(\lambda)$$ here,
with $g_1$ polynomially bounded in the sector 
$-3\pi/2+\epsilon<\arg \lambda<\pi/2-\epsilon.$  The same 
argument gives $Q_2(\lambda)\det S(\lambda e^{i\pi}) /P_2(\lambda)
=\exp g_2(\lambda)$, with $g_2$ polynomially bounded in the same region.
Since $g=g_1-g_2$, we are done.
\end{proof}
 
We continue to use the function $f$ defined in (\ref{eq:f}).  From
the fact that $S(\lambda)$ is unitary when $\arg \lambda=0$ and from Lemma
\ref{l:argchange}, we see that $f(\lambda)$ has neither zeros
nor poles with $\arg \lambda=0$.  Since, for $\tau>0$, $\det S(\tau 
e^{i\pi})\det S(\tau e^{-i\pi})=1$, from Lemma \ref{l:argchange}
the function $f(\lambda)$ has neither
poles nor zeros with $\arg \lambda =-\pi$.

We identify $\{ \lambda \in \Lambda: \; -3\pi/2<\arg \lambda<\pi/2\}$ with 
$\Complex \setminus i[0,\infty)$ and consider the function
$f$ defined in this region, as studied in Lemma \ref{l:frep}.
Define the {\em distribution}
\begin{equation}\label{eq:v}
v(t) = \int _{-\infty}^\infty e^{-i\lambda t} \; \frac{f'(\lambda)}{f(\lambda)} 
d\lambda.
\end{equation}
We clarify that in this integral for $\lambda \in \Real_+$ we understand
$\arg \lambda =0$, and for $\lambda\in \Real_-$ we understand $\arg \lambda
= -\pi$.
Note that this is well-defined as a distribution, as we describe below. 
The function
 $f$ has neither 
zeros nor poles with $\arg \lambda =0$ or $\arg \lambda =\pm \pi$.
Moreover, $f(\lambda)\rightarrow 1$ as $\lambda\rightarrow 0$ with 
$\arg \lambda =0$ or $\arg \lambda = \pm \pi$, see \cite[Section 6]{chobstacle}.
Hence (identifying $\arg \lambda=-\pi$ with $(-\infty,0)$ and
$\arg \lambda =0$ with $(0,\infty)$), we can find a continuous 
function $\ell(f(\lambda))$ on $\Real$ so that $\exp(\ell(f(\lambda)))=f(\lambda)$ for $\lambda\in \Real$.  Moreover, $\ell(f(\lambda))$ is in 
fact smooth on $\Real \setminus \{0\}$, and has an expansion in 
powers of $\lambda$ and $\log \lambda$ at $0$, see \cite[Section 6]{chobstacle}.
Using Lemma \ref{l:frep} we see that $\ell(f)$ is a tempered distribution
on $\Real$.  Hence its derivative, $f'/f$, is also a tempered
distribution.



\begin{lemma}\label{l:vinpoles}
The distribution $v(t)$ defined by (\ref{eq:v}) is given by
$$v(t)= 2\pi i \sum_{ \lambda_j\in \Lambda_{-1}\cap \mcr} 
\left( e^{-i\lambda_j t}+e^{i\overline{\lambda}_j t}\right) - 2\pi i
\sum_{\lambda_j \in \Lambda_0\cap \mcr} 
\left( e^{-i\overline{\lambda_j} t}+e^{i\lambda_j t}\right)\; \text{if $t>0$}.$$
\end{lemma}
\begin{proof}
We use the representation for $f$ from Lemma \ref{l:frep}. 
Hence
\begin{align}\label{eq:logderivexp}
\frac{f'(\lambda)}{f(\lambda)} & = 
\frac{P_1'(\lambda)}{P_1(\lambda)}-\frac{Q_1'(\lambda)}{Q_1(\lambda)}
+\frac{Q_2'(\lambda)}{Q_2(\lambda)} -\frac{P_2'(\lambda)}{P_2(\lambda)} +
g'(\lambda) \nonumber \\ &
= \sum _{\substack{\lambda_j\in \mcr\\ -\pi/2<\arg \lambda_j<3\pi/2}} 
\frac{(\lambda/\overline{\lambda_j})^p}{\lambda-\overline{\lambda_j}} 
- \sum _{\substack{\lambda_j\in \mcr\\ -3\pi/2<\arg \lambda_j<\pi/2}} 
\frac{(\lambda/{\lambda_j})^p}{\lambda-{\lambda_j}}  \nonumber
 \\ & \hspace{4mm}
- \sum _{\substack{\lambda_j\in \mcr\\ -\pi/2<\arg \lambda_j<3\pi/2}} 
\frac{(e^{i\pi} \lambda/{\lambda_j})^p}{e^{i\pi} \lambda-{\lambda_j}} 
+ \sum _{\substack{\lambda_j\in \mcr\\ -3\pi/2<\arg \lambda_j<\pi/2}} \frac{(e^{i\pi} \lambda/\overline{\lambda_j})^p}{e^{i\pi} \lambda-\overline{\lambda_j}} 
+g'(\lambda).
\end{align}

Let $q$ be a polynomial.  
For $t>0$, $a,b\in \Real$ and $b\not = 0$, 
$$\mathcal{F}\left\{ \frac{q(\xi)}{ \xi -(a+ib)}\right\}(t)=
\left\{ \begin{array}{ll} -2\pi i e^{-i(a+ib)t} q(a+ib)\; & 
\text{if $b<0,\; t>0$}\\
0& \text{ if $b>0,\; t>0$}
\end{array}
\right.
$$
as a distribution.  Applying this to (\ref{eq:logderivexp}), we find that

\begin{align} \nonumber 
v(t) & = 2\pi i \left( -\sum_{\lambda_j \in \mcr\cap \Lambda_0} e^{-i\overline{\lambda_j} t}
+  \sum_{\lambda_j \in \mcr\cap \Lambda_{-1}}e^{-i \lambda_j t}
-  \sum_{\lambda_j \in \mcr\cap \Lambda_{0}}e^{i \lambda_j t}
+ \sum _{\lambda_j \in \mcr\cap \Lambda_{-1}}e^{i \overline{\lambda_j} t} \right) 
\\ & \hspace{8mm}
+ \int_{-\infty}^\infty e^{-i\lambda t} g'(\lambda) d\lambda  \;\; \; \text{if $t>0$}.
\end{align}

Now consider $\int_{-\infty}^\infty e^{-i\lambda t}g'(\lambda) d\lambda$, which
is well-defined as a distribution on $\Real$.  
By Lemma \ref{l:frep} the distribution
$\int_{-\infty}^\infty e^{-it\lambda }g(\lambda) d\lambda$ has inverse Fourier
transform which is analytic and polynomially bounded in the
open lower half plane.  
Thus, by a version of the Paley-Weiner-Schwartz Theorem 
(e.g. \cite[Theorem 7.4.3]{ho1}), 
the distribution $\int_{-\infty}^\infty e^{-it\lambda }g(\lambda) d\lambda$
is supported in $t\leq 0$.  Since, in the sense of distributions,
$$\int_{-\infty}^\infty e^{-i\lambda t} g'(\lambda) d\lambda = it \int
_{-\infty}^\infty e^{-i\lambda t} g(\lambda) d\lambda ,$$
the distribution $\int_{-\infty}^\infty e^{-i\lambda t} g'(\lambda) d\lambda$
is supported in $t\leq 0$ as well.
Hence, for $t>0$
$$
v(t)  =  2\pi i \left( -\sum_{\lambda_j \in \mcr\cap \Lambda_0} (
e^{-i\overline{\lambda_j} t} +e^{i \lambda_j t})
+  \sum_{\lambda_j \in \mcr\cap \Lambda_{-1}}(e^{-i \lambda_j t} +
e^{i \overline{\lambda_j} t}) \right),\; t>0.$$
\end{proof}

The following theorem is a Poisson formula for resonances in even dimensions,
complementary to that of \cite{zworskieven}.  The integral appearing here may 
be thought of as an error or remainder term.  Lemma \ref{l:rescontisbig} 
uses Lemma \ref{l:argchange} to bound its contribution in our application,
the proof of Theorem \ref{thm:lowerbd3}.  
\begin{thm} \label{thm:poissonformula}
Let $d$ be even, and $u$ denote the distribution defined by
 (\ref{eq:wavetraceformal}) for a 
self-adjoint black-box type perturbation.
Then, for $t\not = 0$,
\begin{multline*}
u(t) = \sum_{\lambda_j \in \mcr\cap \Lambda_{-1}}\left( e^{-i \lambda_j |t|}
+ e^{i\overline{ \lambda_j} |t|} \right)
+ \sum_{\substack{-\sigma_j^2\in \sigma_{p}(P)\cap(-\infty,0)\\
\sigma_j>0}}(e^{\sigma_j |t|}- e^{-\sigma_j |t|}) \\
+ \sum_{\substack{\mu_l^2 \in \sigma_{p}(P)\cap(0,\infty)\\
\mu_l>0}}(e^{i \mu_l t}+ e^{-i \mu_l t}) 
- \frac{1}{2\pi i} \int_0^\infty \left( e^{-i\lambda |t|} 
\frac{s'(\lambda e^{i\pi})}{s(\lambda e^{i\pi})}+ e^{i\lambda |t|} \frac{s'(\lambda e^{-i\pi})}{s(\lambda e^{-i\pi})}\right) d\lambda  
+m(0).
\end{multline*} 
Here $m(0)$ is the multiplicity of $0$ as a pole of the 
resolvent of $P$, chosen to make the Birman-Krein formula (\ref{eq:birman-krein})
 correct, and $\sigma_p(P)$ is the point spectrum of $P$.
\end{thm}
Before proving the theorem, we note that alternatively we could write
$$ \sum_{\lambda_j \in \mcr\cap \Lambda_{-1}}\left( e^{-i \lambda_j |t|}
+ e^{i\overline{ \lambda_j} |t|} \right)+ \sum_{\substack{\mu_l^2 \in \sigma_{p}(P)\cap(0,\infty)\\
\mu_l>0}}(e^{i \mu_l t}+ e^{-i \mu_l t}) 
=  \sum_{\lambda_j \in \mcr, -\pi<  \arg \lambda_j\leq 0 }\left( e^{-i \lambda_j |t|}
+ e^{i\overline{ \lambda_j} |t|} \right)$$
using that if $\mu_l^2>0$ is an eigenvalue of $P$, then $|\mu_l|$ is 
a pole of $R(\lambda)$ and hence, by our convention, an element of $\mcr$.
\begin{proof} Set $s(\lambda)=\det S(\lambda)$.
By the Birman-Krein formula, 
\begin{multline}\label{eq:birman-krein}u(t) =
\frac{1}{2\pi i} \int_0^\infty (e^{-i\lambda t} + e^{-i\lambda t})
\frac{s'(\lambda)}{s(\lambda)}d\lambda  
+ \sum_{\substack{-\sigma_j^2\in \sigma_{p}(P)\cap(-\infty,0)\\
\sigma_j>0}}(e^{\sigma_j |t|}+ e^{-\sigma_j |t|}) \\ 
+ \sum_{\substack{\mu_l^2 \in \sigma_{p}(P)\cap(0,\infty)\\
\mu_l>0}}(e^{i \mu_l t}+ e^{-i \mu_l t}) 
 +m(0).
\end{multline}
Using the same convention as 
discussed after (\ref{eq:v}) and the definition of $f$ (\ref{eq:f}),
 we write the distribution $v(t)$ from (\ref{eq:v}) as
\begin{align}\label{eq:vexpand}
v(t) & =\int_{-\infty}^\infty e^{-i\lambda t}\left( 
\frac{s'(\lambda)}{s(\lambda)}+ \frac{s'(e^{i\pi} \lambda)}{s(e^{i\pi} \lambda)}
\right)d\lambda \nonumber \\
& = \int_0^\infty 
 ( e^{-i\lambda t } +e^{i\lambda t}) \frac{s'(\lambda)}{s(\lambda) } d\lambda
+ \int_0^\infty \left( e^{-i\lambda t} 
\frac{s'(\lambda e^{i\pi})}{s(\lambda e^{i\pi})}+ e^{i\lambda t} \frac{s'(\lambda e^{-i\pi})}{s(\lambda e^{-i\pi})}\right) d\lambda 
\end{align}
where the second equality follows by a change of variable for the integral
over $(-\infty,0)$. 
The first integral on the right hand side is $2\pi i$ times the first 
term on the right hand side 
in (\ref{eq:birman-krein}).  Solving (\ref{eq:vexpand})
for the integral in (\ref{eq:birman-krein}) and
using Lemma \ref{l:vinpoles} gives, for $t> 0$,
\begin{align} \label{eq:intsum}
\frac{1}{2\pi i} \int_0^\infty ( e^{-i\lambda t } +e^{i\lambda t}) \frac{s'(\lambda)}{s(\lambda) } d\lambda   & 
= \frac{-1}{2\pi i}\int_0^\infty \left( e^{-i\lambda t} 
\frac{s'(\lambda e^{i\pi})}{s(\lambda e^{i\pi})} + e^{i\lambda t} \frac{s'(\lambda e^{-i\pi})}{s(\lambda e^{-i\pi})}\right) d\lambda \nonumber \\ & \hspace{5mm}
+ \sum_{\lambda_j \in \mcr\cap \Lambda_{-1}}\left( e^{-i \lambda_j t}
+ e^{i\overline{ \lambda_j} t} \right) - \sum_{\lambda_j \in \mcr\cap \Lambda_{0}}\left( e^{-i \overline{\lambda_j} t}
+ e^{i \lambda_j t} \right).
\end{align}
 Notice that 
if $\lambda_j\in \Lambda_0\cap \mcr$, then 
$\lambda_j^2 \in \sigma_{p}(P)\cap(-\infty,0)$.  Using this and Lemma 
\ref{l:vinpoles} proves the theorem for $t>0$.  To prove the theorem
for $t< 0$, we note that $u$ is a distribution which is even in $t$.
\end{proof}

\subsection{Comparison of Theorem \ref{thm:poissonformula} to 
other Poisson formulae for resonances}\label{s:comparison}
We briefly compare the result of Theorem \ref{thm:poissonformula} to
earlier Poisson formulae, both in odd and even dimensions.

We note that the proof of Theorem \ref{thm:poissonformula} can, with 
a small modification, be adapted to prove the {\em odd}-dimensional
Poisson formula.  In the  generality of
the  black-box setting we consider here, this is due to  Sj\"ostrand-Zworski
\cite{sj-zwlbII},
but it follows
 earlier work for obstacle scattering by Lax-Phillips
\cite{l-p}, Bardos-Guillot-Ralston \cite{b-g-r}, and Melrose 
\cite{melrosetrace,melrosepoly}
for increasingly large sets of $t\in \Real$.  The proof we describe here
is not very different from, but a bit less direct than, that given in
 \cite{zwpfodd}.  
The value of including this particular variant of the proof of the 
odd-dimensional result is that it shows the consistency of our methods
with the trace formula of 
\cite{b-g-r,l-p,melrosetrace,melrosepoly,sj-zwlbII} in odd dimensions.

Most of our proof of Theorem \ref{thm:poissonformula} is not dimension-dependent.
In fact, the Birman-Krein formula holds in both even and odd dimensions.
We use the distribution $v$ defined in the proof of Theorem \ref{thm:poissonformula}, and note that the computation of $v(t)$, $t>0$ in Lemma \ref{l:vinpoles}
holds in odd dimensions as well, where we make the (natural) identification
of $\Lambda_0$ with the complex upper half plane, and $\Lambda_{-1}$ with
the lower half plane, and use the bound on $g$ proved in \cite{zwpfodd}.
 In {\em odd} dimension, 
\begin{align}
v(t) & =\int_{-\infty}^\infty e^{-i\lambda t}\left( 
\frac{s'(\lambda)}{s(\lambda)}+ \frac{s'(e^{i\pi \lambda})}{s(e^{i\pi} \lambda)}
\right)d\lambda \nonumber \\
& = \int_0^\infty 
 ( e^{-i\lambda t } +e^{i\lambda t}) \frac{s'(\lambda)}{s(\lambda) } d\lambda
+ \int_0^\infty \left( e^{-i\lambda t} 
\frac{s'(\lambda e^{i\pi})}{s(\lambda e^{i\pi})}+ e^{i\lambda t} \frac{s'(\lambda e^{-i\pi})}{s(\lambda e^{-i\pi})}\right) d\lambda  \nonumber \\
& = \int_0^\infty (e^{-i\lambda t } +e^{i\lambda t}) \frac{s'(\lambda)}{s(\lambda) } d\lambda
+ \int_0^\infty ( e^{-i\lambda t} + e^{i \lambda t}) 
\frac{s'(\lambda e^{i\pi})}{s(\lambda e^{i\pi})} d\lambda 
\end{align}
where for the last equality we used $s(\lambda e^{i\pi}) = s(\lambda e^{-i\pi})$
in  odd dimensions.  But in odd dimensions,
$s(\lambda e^{i\pi}) s(\lambda )=1$, so
that $s'(\lambda)/s(\lambda)= 
s'(\lambda e^{\pm i \pi})/ s(\lambda e^{\pm i \pi})$. Hence
for odd dimensions 
$$v(t) = 2 \int_0^\infty (e^{-i\lambda t } +e^{i\lambda t}) \frac{s'(\lambda)}{s(\lambda) } d\lambda.$$
Dividing both sides by $2$ and then continuing to follow the 
proof from the even-dimensional case gives, for $t\not =0$,
\begin{multline} u(t) = \frac{1}{2}  
\sum_{\lambda_j \in \mcr, -\pi<\arg \lambda_j<0} ( e^{-i \lambda_j |t|} + e^{i\overline{\lambda}_j|t|})
+ \sum_{\substack{-\sigma_j^2\in \sigma_{p}(P)\cap(-\infty,0)\\
\sigma_j>0}}e^{\sigma_j |t|}
\\
+  \sum_{\substack{\mu_l^2 \in \sigma_{p}(P)\cap(0,\infty)\\
\mu_l > 0}}(e^{i \mu_l t}+ e^{-i \mu_l t}) + m(0),\; t\not = 0.
\end{multline}
  Noting
that in odd dimensions $\lambda_j$ is a resonance if and only if $-\overline{\lambda}_j$
is a resonance
gives
$$u(t) = \sum_{\lambda_j\in \mcr,\; \lambda_j \not =0}e^{-i\lambda_j |t|}+ m(0),
\; t\not = 0$$
showing the consistency with the odd-dimensional Poisson formula.

Now we return to the case of even dimension $d$.  Theorem 1 of \cite{zworskieven} is, when stated using our notion
\begin{thm}(\cite[Theorem 1]{zworskieven} adapted)\label{thm:zworski} Let 
$d$ be even, $P$ be a self-adjoint
 operator satisfying
the black-box conditions, $u$ the wave trace
defined in (\ref{eq:wavetraceformal}), and $0<\rho<\pi$. 
 Let $s(\lambda)=\det S(\lambda)$,
and $\psi \in C_c^\infty(\Real;[0,1])$ be equal to $1$ near $0$.  Then
\begin{multline*}
u(t) = \sum_{\lambda_j \in \mcr, -\rho/2 <\arg \lambda \leq 0}(e^{-i\lambda_j|t|}
+ e^{i\overline{\lambda_j}|t|}) + 
\sum_{-\sigma_j^2\in\sigma_{p}(P)\cap(-\infty,0), \sigma_j >0}
e^{\sigma_j |t|} + m(0) \\  + 
\frac{1}{\pi i}\int_0^\infty
\psi(\lambda)\frac{s'(\lambda)}{s(\lambda)} \cos(t\lambda)d\lambda 
+ v_{\rho,\psi}(t),\; t\not = 0
\end{multline*}
with 
$$v_{\rho,\psi}\in C^\infty(\Real\setminus\{0\}),\; \partial ^k_t v_{\rho,\psi}(t)=O(t^{-N})\; \forall k,N,\; |t|\rightarrow \infty.$$
\end{thm}
  For the reader comparing
the statement of \cite[Theorem 1]{zworskieven} with this statement, we note
that there are several differences.   One is caused
 by the convention of the location
of the physical half plane (for us, $0<\arg \lambda<\pi$; for 
\cite{zworskieven}, $-\pi<\arg \lambda<0$) and the consequential difference in the location of the resonances.  Another is caused Zworski's convention (see the first paragraph of 
\cite[Section 2]{zworskieven}; the diagram should have a cut extending
along the entire imaginary axis)  {\em defining}, for $\Re \lambda<0$, 
$s'(\lambda)/s(\lambda)= -
\overline{s'(-\overline{\lambda})/s(-\overline{\lambda})}$.
 This means that each resonance with $-\rho/2<\arg 
\lambda_j\leq 0$ actually contributes twice to the sum which appears in
\cite[Theorem 1]{zworskieven}-- once  $e^{-i|t|\lambda_j}$,
and then again  $e^{-(-i|t|\overline{\lambda_j})}$.  

In each of Theorems \ref{thm:poissonformula} and  \ref{thm:zworski},
 any term which does not arise from an eigenvalue
or resonance may be considered part of a ``remainder.''  The remainder terms 
from Theorem \ref{thm:zworski} are smooth away from $t=0$; and one ($v_{\rho,\psi}$) is well-controlled
when $|t|\rightarrow \infty$.   The smoothness of the remainder in \cite[Theorem 1]{zworskieven} means that Zworski's Poisson formula can be used to show that 
if the wave trace
has singularities at a {\em nonzero} time then there  is a lower bound
 on the
number of resonances in sectors near the real axis, see \cite{zworskieven}.
See 
\cite{zworskiremark} for another
application of the Poisson formula of \cite{zworskieven}, also related to the singularities
of the wave trace away from $0$.

The 
integral appearing in Theorem \ref{thm:poissonformula} does 
not, in general, yield a term which is smooth in $t$, even away from $t=0$.
However, the remainder term has the advantage of being in some sense
more explicit than that of Theorem \ref{thm:zworski}.
 As we shall
see in the next section, Lemma \ref{l:argchange} provides enough information
about the remainder in Theorem \ref{thm:poissonformula} to use
the singularity of $u(t)$ at $0$ to prove Theorem \ref{thm:lowerbd3},
and hence Theorems \ref{thm:lowerbd} and \ref{thm:lowerbd2}.

\section{Proof of Theorems \ref{thm:lowerbd} and  \ref{thm:lowerbd2}}
\label{s:proofofthm}

We first prove a more general theorem, Theorem \ref{thm:lowerbd3}
and then show that Theorems \ref{thm:lowerbd}
and \ref{thm:lowerbd2} satisfy the hypotheses of Theorem \ref{thm:lowerbd3}.

\begin{thm} \label{thm:lowerbd3} Let the dimension $d$ be even and 
let $P$ be a black-box operator satisfying the conditions of \cite{sj-zw};
see Section \ref{ss:setup}.
Suppose there is
an $R_1>R_0$ so that for all $b>a>R_1$, if $\chi \in C_c^\infty(\Real^d)$
has support in $\{x\in \Real^d : a<|x|<b\}$ then there is a constant
$C$ depending on $\chi$ so that 
\begin{equation} \label{eq:resolventdecay2}
\| \chi R(\lambda) \chi\| \leq C/\lambda\; \text{ $\lambda 
\in  (1,\infty)$}.
\end{equation}
Let $n_e(r)$ denote the eigenvalues of $P$ of norm at most $r^2$, 
and assume that
$n_e(r) +n_{-1}(r) \leq C'(1+r^d)$ for some $C'>0$. 
Let $u$ be the distribution defined in (\ref{eq:wavetraceformal}).
Suppose there is a constant $\alpha \not = 0$ and $\epsilon_1,\; \epsilon_2>0$
 so that 
\begin{equation}\label{eq:leadingsing}
t^{d-\epsilon_1}\left( u(t) -\alpha|D_t|^{d-1}\delta_0(t) \right) 
\in C^0([0,\epsilon_2]).
\end{equation}
  Then there is a constant $C_0>0$ so that 
$$r^d/C_0\leq n_{-1}(r)+  n_e(r)\; \text{for}\; r\gg 1.$$
\end{thm}

Now we specialize to the case of $P$ as in the statement of Theorem 
\ref{thm:lowerbd3}.  

Set, for $t \not = 0$,
\begin{equation}\label{eq:w}
w(t)=u(t)- \sum_{\lambda_j \in \mcr\cap \Lambda_{-1}}\left( e^{-i \lambda_j |t|}
+ e^{i\overline{ \lambda_j} |t|} \right)-  \sum_{\substack{\mu_l^2 \in \sigma_{p}(P)\cap(0,\infty)\\
\mu_l>0}}(e^{i \mu_l t}+ e^{-i \mu_l t}) .
\end{equation}

\begin{lemma}\label{l:rescontisbig}
Let $\phi \in C_c^\infty(\Real_+)$
and set $\phi_\gamma(t)= \frac{1}{\gamma}\phi \left( \frac{t}{\gamma}\right)$, $\gamma >0$.
Then with $w$ as defined by 
(\ref{eq:w}) there is a constant $C>0$ so that 
$$\left| \int w(t) \phi_\gamma(t) dt\right| \leq C \gamma^{-(d-1)}\; 
\text{for $\gamma\in (0,1]$}.$$
\end{lemma}
\begin{proof}
We use Theorem \ref{thm:poissonformula}.  We note that 
 $\int \phi_\gamma (t) m(0)dt= 
m(0)\int \phi(t)dt$ is independent of $\gamma$.   Moreover,
\begin{align}\label{eq:eigenvaluecont}
 \left|  \int \sum_{\substack{-\sigma_j^2\in \sigma_{p}(P)\cap(-\infty,0)\\
\sigma_j>0}}(e^{\sigma_j |t|}- e^{-\sigma_j |t|}) \phi_\gamma(t) dt \right|
& = \left|  \sum_{\substack{-\sigma_j^2\in \sigma_{p}(P)\cap(-\infty,0)\\\sigma_j>0}}
\int_0^\infty \phi(t) ( e^{\gamma \sigma_j t}-e^{-\gamma \sigma_j t}) dt \right|\\
\nonumber
& 
\leq C\; \text{for $\gamma \in (0,1]$}.
\end{align}

It remains to bound the term corresponding to the integral appearing in the 
Poisson formula of Theorem \ref{thm:poissonformula}.
Since for $\arg \lambda=0$ $\det S(\lambda e^{i\pi})\not =0$, 
there is a differentiable
 function $g_1$ defined on $(0,\infty)$ so that $s(\lambda e^{i\pi})= e^{g_1(\lambda)}$
when $\arg \lambda =0$.  This does not uniquely determine $g_1$.
Since $\lim _{\lambda \downarrow 0}s(\lambda e^{i\pi})=1$, we may choose $g_1$ 
to satisfy $\lim _{\lambda \downarrow 0}g_1(\lambda)=0$.
Using the relation $S^*(\overline{\lambda})S(\lambda)=I$, 
$$\frac{s'(\lambda e^{i\pi})}{s(\lambda e^{i\pi})}=-\overline{ \left( \frac{s'(\lambda e^{-i\pi})}{s(\lambda e^{-i\pi})}\right) }\; \text{ if $\arg \lambda =0$}.$$
Hence 
$$\frac{s'(\lambda e^{i\pi})}{s(\lambda e^{i\pi})}= - g_1'(\lambda),\;
\text{if}\;  \arg \lambda =0 \; \text{and}\;
\frac{s'(\lambda e^{-i\pi})}{s(\lambda e^{-i\pi})}= 
\overline{g}_1'(\lambda),
\; \text{if}\;  \arg \lambda =0.$$
Thus, for $t>0$
\begin{align}\label{eq:intwithg1}
 \int_0^\infty
\left( e^{-i\lambda t} 
\frac{s'(\lambda e^{i\pi})}{s(\lambda e^{i\pi})}+ e^{i\lambda t} \frac{s'(\lambda e^{-i\pi})}{s(\lambda e^{-i\pi})}\right) d\lambda 
=- \int_0^\infty
\left( e^{-i\lambda t} g'_1(\lambda)- e^{i\lambda t} \overline{g}_1'(\lambda) \right) d\lambda
\end{align}
For $\tau \in \Real$, set
$$g_2(\tau) = \left\{ \begin{array}{ll}
g_1(\tau)& \text{if $\tau>0$} \\
\overline{g}_1(-\tau) & \text{if $\tau<0$} \\
0& \text{if $\tau=0$}.
\end{array}
\right.
$$
Note that $g_2$ is continuous, with 
$$g_2'(\tau)= \left\{ \begin{array}{ll}
g_1'(\tau) & \text{if $\tau >0$}\\
-\overline{g}_1'(-\tau) & \text{if $\tau <0$.}
\end{array}
\right.
$$
From (\ref{eq:intwithg1}) and using the continuity of $g_2$,
$$\int_0^\infty
\left( e^{-i\lambda t} 
\frac{s'(\lambda e^{i\pi})}{s(\lambda e^{i\pi})}+ e^{i\lambda t} \frac{s'(\lambda e^{-i\pi})}{s(\lambda e^{-i\pi})}\right) d\lambda =-
\int_{-\infty}^\infty e^{-i \tau t} g_2'(\tau) d\tau =- \widehat{g'_2}(t).$$
Thus
\begin{align*} 
\int_{-\infty}^\infty \phi_\gamma(t) \int _0^\infty
\left( e^{-i\lambda t} 
\frac{s'(\lambda e^{i\pi})}{s(\lambda e^{i\pi})}+ e^{i\lambda t} \frac{s'(\lambda e^{-i\pi})}{s(\lambda e^{-i\pi})}\right) d\lambda dt 
& = -\int_{-\infty}^\infty \phi_\gamma(t) \widehat{g_2'}(t) dt \\
& = -\int_{-\infty}^\infty \hat{\phi} (\gamma \tau) g_2'(\tau) d\tau\\
& = \int_{-\infty}^\infty\gamma \hat{\phi}' (\gamma \tau) g_2(\tau) d\tau.
\end{align*}
By Lemma \ref{l:argchange}, there is a constant $C$ so that 
$|g_1(\lambda)|\leq C(1+|\lambda|^{d-1})$ for $\arg \lambda =0$. Thus
\begin{align}
\left| \int_{-\infty}^\infty\gamma \hat{\phi}' (\gamma \tau) g_2(\tau) d\tau
\right| & 
= \left| \int_{-\infty}^\infty \hat{\phi}' ( \tau) g_2(\tau/\gamma) d\tau
\right|\nonumber \\
& \leq \int_{-\infty}^\infty C(1+|\tau |)^{-d-1} (1+|\tau/\gamma|)^{d-1}d\tau
\nonumber \\
& \leq C \gamma^{-(d-1)} \int_{-\infty}^\infty (1+|\tau |)^{-2}d\tau \leq C 
\gamma^{-(d-1)}.
\end{align}
Together with Theorem \ref{thm:poissonformula},
  (\ref{eq:eigenvaluecont}),  and the boundedness of the contribution of $m(0)$, this proves the lemma. 
\end{proof}

\begin{lemma}\label{l:singnear0} Let $d$ be even,
and let $P$ and $u$ satisfy the hypotheses
of Theorem \ref{thm:lowerbd3}.
  Let $\phi\in C_c^{\infty}(\Real_+)$, $\phi \geq 0$, 
$\phi(1)\not =0$.  Then there are constants $c > 0$ and $\gamma_0>0$
so that 
$$\left| \int \phi_\gamma(t) u(t) dt \right| \geq c \; |\alpha| \gamma^{-d} \;
 \text{if $\gamma\in (0,\gamma_0]$}$$
where $\alpha$ is as in (\ref{eq:leadingsing}).
\end{lemma}
\begin{proof}
We recall the assumption that there are $\epsilon_1,\; \epsilon_2>0$ so that
\begin{equation}\label{eq:at0first}
t^{d-\epsilon_1}\left( u(t) -\alpha|D_t|^{d-1}\delta_0(t) \right) \in
 C^0([0,\epsilon_2]).
\end{equation}
There is a constant $b\not = 0$ so that for $t>0$, $|D_t|^{d-1}\delta_0(t)=
bt^{-d}$. Thus
\begin{equation}\label{eq:big}
\left| \int_{-\infty}^\infty \phi_\gamma(t)|D_t|^{d-1}\delta_0(t) dt \right|
= \left| b \int_0^\infty \frac{1}{\gamma} \phi(t/\gamma)t^{-d} dt\right|
= \gamma^{-d }\left| b\int_0^\infty \phi(t) t^{-d}  dt \right|
\geq c' \gamma^{-d}
\end{equation}
with $c'>0$.

Set $w_r(t)= u(t) -\alpha|D_t|^{d-1}\delta_0(t)$.
Since by (\ref{eq:at0first})
 $t^{d-\epsilon_1} w_r(t)$ is continuous for $t\in [0,\epsilon_2]$,
there is a constant $C$ so that $| \phi_\gamma(t) w_r(t)|\leq C t^{-(d-\epsilon_1)}
\phi_\gamma(t)$
when $\gamma>0$ is sufficiently small.
Thus for $\gamma>0$ sufficiently small,
 \begin{equation}\label{eq:smaller}
\left| \int \phi_\gamma(t) w_r(t)dt\right|
\leq \int C t^{-(d-\epsilon_1)}\phi_\gamma(t)dt \leq C \gamma^{-(d-\epsilon_1)}\int \phi_\gamma (t)dt
\leq C \gamma^{-d+\epsilon_1}.
\end{equation}
The lemma follows from (\ref{eq:at0first}), (\ref{eq:big}), and (\ref{eq:smaller}),
by choosing $\gamma_0>0$ sufficiently small.
\end{proof}

\vspace{2mm}
\noindent
{\em Proof of Theorem \ref{thm:lowerbd3}}.  
The theorem now follows almost immediately from an application of 
\cite[Propostion 4.2]{sj-zwlbII}.  Here we use, in the 
notation of that proposition,  $V(r)=r^{d}$.
\footnote{This $V$ has no relation to the potential $V$ in the statement
of Theorem \ref{thm:lowerbd2}.}  We use our assumption 
that $n_{-1}(r) =O(r^d)$ and and $n_e(r)=O(r^d)$ as $r\rightarrow \infty$.
   By Lemmas 
\ref{l:rescontisbig} and \ref{l:singnear0} the other conditions of the 
proposition are satisfied.
\qed

\vspace{2mm}
\noindent
{\em Proof of Theorems \ref{thm:lowerbd} and \ref{thm:lowerbd2}}.  If $M$ has no boundary,
then there is an $\epsilon>0$ and a  constant $\tilde{c}_d\not = 0 $ so that 
\begin{equation}\label{eq:at0}
t^{d-1}\left( u(t) -\tilde{ c}_d (\vol K -\vol B(0;R_0))|D_t|^{d-1}\delta_0(t) \right) \in C^\infty([0,\epsilon]),
\end{equation}
see \cite{hosf}.  In case $M$ has a boundary (in particular, if 
$M=\Real^d\setminus \mco$), that  (\ref{eq:at0}) holds follows from 
\cite{ivrii, melrose} or 
\cite[Prop. 29.1.2 and the proof of Prop. 29.3.3]{ho4}.

By results of Burq \cite[(8.5)]{burq} for the case of $\Real^d \setminus \mco$
or Cardoso-Vodev \cite{c-v} for $P=-\Delta_g$ on $M$, we have that 
(\ref{eq:resolventdecay2}) holds.
We note that with the assumptions we have made on  $\mco$ 
and $M$ the operator $P$ has 
no positive eigenvalues, and only finitely many negative eigenvalues.  
Using the upper bound of \cite{vodeveven, vodev2},
$n_{-1}(r)=O(r^d)$.  Then
  Theorem 
\ref{thm:lowerbd2} (which implies Theorem \ref{thm:lowerbd})
follows immediately from Theorem \ref{thm:lowerbd3}.
\qed

\end{document}